\documentclass[11pt,a4paper]{amsart}
\usepackage{amssymb}
\usepackage{amscd}

\theoremstyle{plain}
\newtheorem{thm}{Theorem}[section]
\newtheorem{lem}[thm]{Lemma}
\newtheorem{pro}[thm]{Proposition}
\newtheorem{cor}[thm]{Corollary}
\theoremstyle{definition}

\newcommand{\Z}{\mathbb{Z}}

\DeclareMathOperator{\Ext}{Ext}

\DeclareMathOperator*{\rightoverright}{\parbox{2em}{\centerline{$\longrightarrow$}\vskip -6pt\centerline{$\longrightarrow$}}}

\begin{document}
\title{$\Sigma$-pure-injective modules for string algebras and linear relations}
\author{Raphael Bennett-Tennenhaus and William Crawley-Boevey}

\address{Fakult\"at f\"ur Mathematik, Universit\"at Bielefeld, 33501 Bielefeld, Germany}
\email{rbennett@math.uni-bielefeld.de, wcrawley@math.uni-bielefeld.de}



\thanks{Both authors are supported by the Alexander von Humboldt Foundation 
in the framework of an Alexander von Humboldt Professorship 
endowed by the German Federal Ministry of Education and Research.}

\begin{abstract}
We prove that indecomposable $\Sigma$-pure-injective modules for a string algebra are string or band modules.
The key step in our proof is a splitting result for infinite-dimensional linear relations.
\end{abstract}
\maketitle

\section{Introduction}

A \emph{string algebra} is one of the form $\Lambda=KQ/(\rho)$
where $K$ is a field, $Q$ is a quiver, $KQ$ is the path
algebra, and $(\rho)$ denotes
the ideal generated by a set $\rho$ of paths of length at least 2, satisfying
\begin{itemize}
\item[(a)]
any vertex of $Q$ is the head of at most two arrows and the tail
of at most two arrows, and 
\item[(b)]
given any arrow $y$ in $Q$, there is at most one path $xy$
of length $2$ with $xy\notin\rho$ and at most one path $yz$ of
length $2$ with $yz\notin\rho$.
\end{itemize}
For simplicity we suppose that $Q$ has only finitely many vertices (so is finite), so that the
algebra $\Lambda$ has a unit element.

It is well-known that the finite-dimensional indecomposable modules for a string algebra 
are classified in terms of strings and bands, see for example \cite{BR,CB}.
It is also interesting to study infinite-dimensional modules, especially pure-injective 
modules, see \cite{R3,PP,PP2}. 
In this paper we classify indecomposable $\Sigma$-pure-injective modules for string algebras. 
Recall that a module is said to be \emph{pure-injective} or \emph{algebraically compact} if it is injective 
with respect to pure-exact sequences (where an exact sequence is \emph{pure-exact} if it 
remains exact after tensoring with any module). 
A module is \emph{$\Sigma$-pure-injective} if any direct sum of 
copies of it is pure-injective. 
There are many equivalent formulations, see for example \cite[\S 4.4.2]{P}. 
Note that any countable-dimensional pure-injective module is $\Sigma$-pure-injective, 
see \cite[Corollary 4.4.10]{P}.

Associated to a string algebra $\Lambda$ there are certain \emph{words} whose letters are the arrows of $Q$
and their inverses. The words may be finite or (as in \cite{R3,CB}) infinite. 
Associated to such a word $C$ there is a module $M(C)$. (We recall the appropriate definitions in \S\ref{s:string}). 
By a \emph{string module} one means a module $M(C)$ with $C$ not a periodic word. If $C$ is periodic, then
$M(C)$ becomes a $\Lambda$-$K[T,T^{-1}]$-bimodule, and given any indecomposable $K[T,T^{-1}]$-module $V$
there is a corresponding \emph{band module} $M(C,V) = M(C)\otimes_{K[T,T^{-1}]} V$.
It is known that string modules are indecomposable, 
and Harland \cite{Har} has given a criterion in terms of a word $C$, for when the string module $M(C)$
is $\Sigma$-pure-injective; for convenience we recall his criterion in \S\ref{s:string}.
Our main result is as follows.

\begin{thm}
\label{t:mainstringtheorem}
Every indecomposable $\Sigma$-pure-injective module
for a string algebra $\Lambda$ is either a string module $M(C)$ or a band module $M(C,V)$
with $V$ a $\Sigma$-pure-injective $K[T,T^{-1}]$-module.
\end{thm}

The indecomposable $\Sigma$-pure-injective $K[T,T^{-1}]$-modules are the 
indecomposable finite-dimensional modules,
the \emph{Pr\"ufer} modules, which are the injective envelopes of the simple modules, 
and the function field $K(T)$.
It is easy to see that the corresponding $\Lambda$-modules $M(C,V)$ are also $\Sigma$-pure-injective,
for example using \cite[Theorem 4.4.20(iii)]{P}.
Since any $\Sigma$-pure-injective module is a direct sum of indecomposables, the theorem,
combined with \cite[Theorem 9.1]{CB}, implies that $M(C,V)$ is indecomposable for $V$ indecomposable 
$\Sigma$-pure-injective.

The proof of our theorem uses the functorial filtration method, which
goes back to the classification
of Harish-Chandra modules for the Lorenz group by Gelfand and Ponomarev \cite{GP},
and was used for the classification of finite-dimensional modules 
for string algebras by Butler and Ringel \cite{BR}. 
The method depends on a certain splitting result for finite-dimensional linear relations, 
see \cite[Theorem 3.1]{GP}, \cite[\S 2]{R1} and \cite[\S7]{Gab}.
An extension of this splitting result to some infinite-dimensional 
relations was obtained in \cite[Lemma 4.6]{CB}. 
A key step in the proof of our theorem is the generalization of
this splitting result to the $\Sigma$-pure-injective case, 
which we now explain.

Fix a base field~$K$.
A \emph{linear relation} $(V,C)$ consists of a vector space~$V$ 
and a subspace $C$ of $V\oplus V$.
The category of linear relations has as morphisms
$(V,C)\to (U,D)$ the linear maps $f:V\to U$
with the property that $(f(x),f(y))\in D$ for all $(x,y)\in C$. 
Any linear relation $(V,C)$ defines a Kronecker module
\[
X\rightoverright^{p}_{q} Y
\]
where $X=C$, $Y = V$ and $p$ and $q$ are the first and second projections, 
and in this way the 
category of linear relations is equivalent to the full 
subcategory of the category of Kronecker modules,
consisting of those modules such that the map 
$(\begin{smallmatrix} p\\ q\end{smallmatrix}) : X\to Y^2$ is injective.
Linear relations can be considered as generalizations of linear maps, 
and one defines $Cu = \{ v\in V \colon (u,v)\in C \}$ for $u\in V$ 
and $CU = \bigcup_{u\in U} Cu$ for $U\subseteq V$. 
If $U$ is a subspace of $V$ and $C$ is a relation on $V$, 
then $C|_U$ denotes $C\cap (U\oplus U)$.

Given a linear relation $(V,C)$, we recall \cite[Definition 4.3]{CB} 
that there are subspaces of $V$ defined by
\[
\begin{split}
C^\sharp &= \{ v\in V : \text{$\exists v_n\in V$ for all $n\in \Z$ with $v_{n+1}\in C v_n$ and $v=v_0$}\},
\\
C^\flat &= C_+ + C_-,
\ C_\pm = \{ v\in V : \text{$\exists v_n\in V$ as above 
with $v_n = 0$ for $\pm n\gg0$ }\}.
\end{split}
\]
By \cite[Lemma 4.5]{CB} the quotient $C^{\sharp}/C^{\flat}$ is a $K[T,T^{-1}]$-module with the action of $T$ given by $T(C^{\flat}+v)=C^{\flat}+w$ if and only if $w\in C^{\sharp}\cap(C^{\flat}+Cv)$. 
Using \cite[Lemma 4.6]{CB} we prove the following.

\begin{thm}
\label{t:mainrelationthm}
As Kronecker modules, $(C^\flat,C|_{C^\flat})$ and $(C^\sharp,C|_{C^\sharp})$ are both pure submodules of $(V,C)$.
\end{thm}

We say that a relation $(V,C)$ is \emph{automorphic} if both projection maps $p,q:C\to V$ are isomorphisms.
The theorem implies our splitting result for linear relations.

\begin{cor}
\label{c:maincor}
If $(V,C)$ is $\Sigma$-pure-injective as a Kronecker module, 
then there is a decomposition
$C^\sharp = C^\flat \oplus U$ 
such that $(U,C|_U)$ is an automorphic relation. Moreover $C^{\sharp}/C^{\flat}$ is a $\Sigma$-pure-injective $K[T,T^{-1}]$-module.
\end{cor}

\section{Linear relations}
Products $CD$ and inverses $C^{-1}$ of relations on $V$ are defined by $u\in CDv$ if $u\in Cw$ and $w\in Dv$ for some $w\in V$, and $u\in C^{-1}v$ $\Leftrightarrow$ $v\in Cu$. Recall \cite{CB} that
\[
\begin{split}
C' &= \bigcup_{n=0}^\infty C^n 0, \text{ and}
\\
C'' &= \{v_0 \in V : \text{$\exists v_n \in V$ for $n>0$ with $v_n \in Cv_{n+1}$ for all $n\ge 0$} \}, 
\end{split}
\]
so that $C^\sharp = C'' \cap (C^{-1})''$, $C_+ = C'' \cap (C^{-1})'$, $C_- = (C^{-1})'' \cap C'$.

\begin{lem}
If $(V,C)$ is automorphic, then $C^\flat = 0$ and $C^\sharp = V$.
\end{lem}

\begin{proof}
Clear.
\end{proof}

\begin{lem}
\label{l:straightforward}
If $C$ is a relation, then 
$(C|_{C^\flat})^\flat = C^\flat$ and $(C|_{C^\sharp})^\sharp = C^\sharp$.
\end{lem}

\begin{proof}
Straightforward.
\end{proof}

The category of linear relations inherits an exact structure from the category of Kronecker modules,
in which a sequence of relations
\[
0 \to (V_1,C_1) \xrightarrow{f} (V_2,C_2) \xrightarrow{g} (V_3, C_3)\to 0
\]
is exact provided that 
$0\to V_1\to V_2\to V_3\to 0$ and
$0\to C_1\to C_2\to C_3\to 0$ are exact.

\begin{lem}
\label{l:flatsharpexact}
Given a relation $(V,C)$, there is an exact sequence
\[
0\to (C^\flat,C|_{C^\flat}) \to 
(C^\sharp,C|_{C^\sharp}) \to 
(C^\sharp/C^\flat,(C|_{C^\sharp})/(C|_{C^\flat}))
\to 0
\]
where the third term is automorphic.
\end{lem}

\begin{proof}
We need to show that the third term is automorphic.
Consider the map $C|_{C^\sharp}\to C^\sharp/C^\flat$ given by the first projection, say.

The map is onto since by definition any element $v_0$ of $C^\sharp$ belongs to an infinite sequence of elements $v_n\in V$
with $(v_{n+1},v_n)\in C$ for all $n$, and then $(v_0,v_{-1})\in C|_{C^\sharp}$.

The kernel of the map is the set of pairs $(x,y)\in C$ with $x,y\in C^\sharp$ 
and $x\in C^\flat$.
But then $y \in C^\sharp \cap CC^\flat$, and by \cite[Lemma 4.4]{CB} this is equal to $C^\flat$, so
the kernel is $C|_{C^\flat}$.
\end{proof}

A relation $(V,C)$ is said to be \emph{split} provided that
there is a subspace $U$ of $V$ such that $C^\sharp = C^\flat \oplus U$ and $(U,C|_U)$ 
is an automorphic relation \cite[\S 4]{CB}.

\begin{lem}
\label{l:splitequiv}
A relation $(V,C)$ is split if and only if the exact sequence in Lemma~\ref{l:flatsharpexact} is split.
\end{lem}

\begin{proof}
It suffices to show that if $(V,C)$ is split, then $C|_{C^\sharp} = C|_{C^\flat} \oplus C|_U$, for then
$(U,C|_U)$ is a complement for $(C^\flat,C|_C^\flat)$ as Kronecker modules.
Suppose $(x,y)\in C|_{C^\sharp}$. 
Write $x = z+u$ with $z\in C^\flat$ and $u\in U$. By assumption there is $w\in U$ with $(u,w)\in C$.
Since $C$ is linear, $(z,y-w)\in C$. Thus $y-w \in Cz \subseteq CC^\flat$. But also 
$y-w \in C^\sharp$. Thus $y-w \in C^\flat$ by \cite[Lemma 4.4]{CB}.
Then $(x,y) = (u,w) + (z,y-w) \in C|_U + C|_{C^\flat}$.
\end{proof}

\begin{lem}
\label{l:exactprops}
Consider an exact sequence of relations
\[
0 \to (V_1,C_1) \xrightarrow{f} (V_2,C_2) \xrightarrow{g} (V_3, C_3)\to 0
\]
where we identify $V_1$ as a subspace of $V_2$. Then
\begin{itemize}
\item[(i)]
if $C_1^\sharp = V_1$ and $C_3^\sharp = V_3$ then $C_2^\sharp = V_2$;

\item[(ii)]
if $C_1^\flat = V_1$ and $C_3^\flat = V_3$ then $C_2^\flat = V_2$; and

\item[(iii)]
if $C_1^\flat=V_1$ and $C_3^\flat = 0$ then $C_2^\flat = V_1$.
\end{itemize}
\end{lem}

\begin{proof}
(i) By symmetry, it suffices to show that if $v\in V_2$ then $v\in C_2 V_2$.
By assumption $g(v) \in V_3 = C_3^\sharp$, so $g(v) \in C_3 V_3$.
Thus $(u,g(v))\in C_3$ for some $u\in V_3$.
Since the map $C_2\to C_3$ is onto, there is $(x,y)\in C_2$ with $g(x) = u$ and $g(y) = g(v)$.
Then $g(y-v) = 0$, so we can identify $y-v$ as an element of $V_1 = C_1^\sharp$, so $y-v \in C_1 V_1$, so there is $w\in V_1$ with $(w,y-v)\in C_1$.
But then $(x-w,v) \in C_2$, so $v\in C_2 V_2$, as required.

(ii) We show by induction on $n$ that if $v\in V_2$ and $g(v)\in C_3^n 0$ then $v\in C_2^\flat$.
The result then follows by symmetry, using that $g$ is onto.
If $n=0$ then $g(v)=0$, so $v\in V_1 = C_1^\flat \subseteq C_2^\flat$.
If $n>1$, then $g(v) \in C_2 w$ with $w\in C_2^{n-1} 0$. Now since the map $C_2\to C_3$ is onto,
there is $(x,y)\in C_2$ with $(g(x),g(y)) = (w,g(v))$.
By induction $x\in C_2^\flat$. Then $y\in C_2 x \subseteq C_2 C_2^\flat$, and $y\in C_2^\sharp$,
so $y\in C_2^\flat$. Also $g(v)=g(y)$, so $v-y\in V_1 = C_1^\flat \subseteq C_2^\flat$, so
$v\in C_2^\flat$.

(iii) Clearly $V_1 = C_1^\flat \subseteq C_2^\flat$.
Conversely, if $v\in C_2^\flat$, then $g(v)\in C_3^\flat$, so $g(v)=0$, so $v\in V_1$.
\end{proof}

We recall the classification of Kronecker modules, see for example \cite{Bu}.
If $M$ is a finite-dimensional indecomposable Kronecker module, say of the form
\[
X\rightoverright^p_q Y,
\]
then either it is automorphic regular, meaning that $p$ and $q$ are isomorphisms, 
or $M$ is of one of the following types, where $X$ has basis $(x_i \colon i\in I)$, 
$Y$ has basis $(y_j \colon j\in J)$, $p(x_i) = y_i$ (or 0 if $i\notin J$) 
and $q(x_i) = y_{i+1}$ (or 0 if $i+1\notin J$).

\begin{itemize}
\item[(i)]
Preprojectives $P_n$ ($n\ge 0$): $I=\{1,\dots,n\}$, $J = \{1,\dots,n+1\}$.
\item[(ii)]
Preinjectives $I_n$ ($n\ge 0$): $I=\{0,\dots,n\}$, $J = \{1,\dots,n\}$.
\item[(iii)]
0-Regulars $Z_n$ ($n\ge 1$): $I=\{1,\dots,n\}$, $J=\{1,\dots,n\}$.
\item[(iv)]
$\infty$-Regulars $R_n$ ($n\ge 1$): $I=\{0,\dots,n-1\}$, $J=\{1,\dots,n\}$.
\end{itemize}
Linear relations correspond to Kronecker modules without $I_0$ as a direct summand.

\begin{lem}
\label{l:nomaps}
Let $(V,C)$ be a linear relation, let $U$ be one of the following subspaces of $V$
and let $M$ be a finite-dimensional indecomposable Kronecker module of the indicated type:
\begin{itemize}
\item[(i)] $U=C^{\flat}$ and $M$ is preinjective, or
\item[(ii)] $U=C^{\sharp}$ and $M$ is preinjective, or
\item[(iii)] $U=C^{\sharp}$ and $M$ is automorphic regular.
\end{itemize}
Then there is no non-zero map of Kronecker modules
$\psi:M\rightarrow (V/U,C/C|_{U})$.
\end{lem}

\begin{proof}
(i), (ii) For $M=I_n$ the map $\psi$ consists of maps 
$\theta:X\to C/C|_U$
and $\phi:Y\to V/U$, sending
$x_i$ to the coset of $(v_i',v_{i+1}'')\in C$ for $0\le i\le n$
and $y_j$ to the coset of $v_j$ for $1\le j\le n$,
and such that $v_i'-v_i, v_{i+1}-v_{i+1}''\in U$ for $0\le i\le n$,
where $v_0 = v_{n+1} = 0$. Note $U \subseteq (C^{-1})''$.

We claim that all $v_i',v_{i+1}'' \in (C^{-1})''$. This is true for
$v_{n+1}''$; if true for $v_{i+1}''$ it follows for $v_i'$
since $v_i' \in C^{-1}v_{i+1}''$; and if true for $v_i'$ it follows for
$v_i''$ since $v_i'-v_i'' \in U \subseteq (C^{-1})''$.
The claim follows.

Dually, starting with $v_0'$, we see that all $v_i',v_{i+1}'' \in C''$.
Thus all $v_i',v_{i+1}''\in C^\sharp$. If $U=C^{\sharp}$ then $v_j\in U$ for $1\le j\le n$ in which case $\theta=\phi=0$. So we may assume $U=C^{\flat}$. 

Now we claim that all $v_i',v_{i+1}'' \in C^\flat$.
This is true for $v_0'$; 
if true for $v_i'$ it follow for $v_{i+1}''$
since $v_{i+1}'' \in C^\sharp \cap Cv_i' \subseteq C^\sharp \cap CC^\flat \subseteq C^\flat$ by \cite[Lemma 4.4]{CB};
if true for $v_i''$ it follows for $v_i'$ since $v_i'-v_i''\in C^\flat$. Thus $\psi=0$ as above.

(iii) Let $x_1,\dots,x_n$ be a basis for $X$, and so  $y_1,\dots,y_n$ is a basis for $Y$ where $y_i=p(x_i)$. There is an invertible matrix $A=(a_{ij})$ with $a_{ij}\in K$ and $q(x_i)=\sum^n_{j=1}a_{ij}y_j$. The map $\psi$ consists of 
$\theta':X\to C/C|_{C^{\sharp}}$
and $\phi':Y\to V/C^{\sharp}$, sending
$x_i$ to the coset of $(w_i,w_i')$ and $y_i$ to the coset of $w_i''$, such that $w_i-w_i'', w_i'-\sum^n_{j=1}a_{ij}w_j''\in C^{\sharp}$. It suffices to show $w_i\in C^{\sharp}$. 

Note $w_i'-\sum^n_{j=1}a_{ij}w_j \in C^{\sharp}$ since this is the sum of $\sum^n_{j=1}a_{ij}(w_j''-w_j)$ and $w_i'-\sum^n_{j=1}a_{ij}w_j''$. By \cite[Lemma 4.4]{CB} we have $C^{\sharp} \subseteq C^{-1}C^{\sharp}$ and so there is some $u_i\in C^{\sharp}$ for which $(u_i,w_i'-\sum^n_{j=1}a_{ij}w_j)\in C$. Thus we have $(w_i-u_i,\sum^n_{j=1}a_{ij}w_j)\in C$.

Since $u_i\in C^{\sharp}$ there exist $u_{i,t} \in C^{\sharp}$ for $t\in \mathbb{Z}$ such that $u_{i,0}=u_i$ and $u_{i,t} \in Cu_{i,t-1}$ for all $t$. For $1\leq i,j\leq n$ let $a_{ij}^{+}:=a_{ij}$ and let $a_{ij}^{-}$ be the $(i,j)^{\mathrm{th}}$ entry of the matrix $A^{-1}$. We define elements $w_{i}^s,u_{i,t}^s \in V$ iteratively as follows. Let $w^0_i=w_i$ and $u^0_{i,t}=u_{i,t}$, and for $d \geq 1$ let
\[
\begin{array}{cc}
w_{i}^{\pm d}=\sum_{j=1}^{n}a^{\pm}_{ij}w_{j}^{\pm(d-1)} &  u_{i,t}^{\pm d}=\sum_{j=1}^{n}a^{\pm}_{ij}u_{j,t}^{\pm(d-1)}
\end{array}
\]
By construction $(w_i^{d}-u_{i,0}^d,w_i^{d+1})\in C$ when $d=0$. If this true for some $d\geq 0$ then 
\[
\begin{array}{c}
(w_i^{d+1}-u^{d+1}_{i,0},w^{d+2}_i)=\sum_{j=1}^{n}a_{ij}(w_j^{d}-u^{d}_{j,0},w^{d+1}_j)\in C,
\end{array}
\]
hence for all $d\geq 0$ we have $(w_i^{d}-u_{i,0}^d,w_i^{d+1})\in C$. Note that $(u_{i,t}^d,u^d_{i,t+1})\in C$ for all $t\in \mathbb{Z}$.
We claim $(z^d_i,z^{d+1}_i)\in C$ for all $d\geq 0$ where $z_i^0=w_i^0-u_{i,0}^0$, $z_i^1=w_i^1$ and  $z^d_i=w^d_{i}+\sum_{r=1}^{d-1}u_{i,r}^{d-r}$ for $d\geq 2$. For $d=0$ the claim holds by construction. If $(z^{d-1}_i,z^d_i)\in C$ for some $d\geq 1$ then
\[
\begin{array}{c}
z_i^{d+1}=w^{d+1}_{i}+\sum_{r=1}^{d}u_{i,r}^{d+1-r}\in C(w_i^{d}-u_{i,0}^d+\sum_{r=1}^{d}u_{i,r-1}^{d+1-r})
\end{array}
\]
by the above, and as $\sum_{r=2}^{d}u_{i,r-1}^{d+1-r}=\sum_{r=1}^{d-1}u_{i,r}^{d-r}$ this gives $(z^d_i,z^{d+1}_i)\in C$. Now let $z_i^{d}=w_i^d+\sum_{r=d}^{0}u_{i,r}^{r-d}$ for $d\leq1$. As above we have $(z^d_i,z^{d+1}_i)\in C$ for $d\leq 0$, and so altogether we have $z_i^0=w_i-u_i \in C^{\sharp}$, as required. 
\end{proof}

\begin{lem}
\label{l:somesplit}
Let $(V,C)$ be a relation with $V=C^\sharp$, and let $M$ be a finite-dimensional indecomposable preprojective, 0-regular or $\infty$-regular Kronecker module. Then  $\Ext^1(M,(V,C))=0$.
\end{lem}

\begin{proof}
We can reduce to the case $M = R_1$ or $Z_1$, since any
$M$ as listed is an iterated extension of copies of $R_1$ or $Z_1$
and possibly also the projective module $P_0$. 
By symmetry we reduce to $M=R_1$.
Consider an extension
\[
0\to (V,C)\to (W,D) \to M\to 0
\]
and identify $V$ as a subspace of $W$, so $C$ is a subspace of $D$.
Let $w\in W$ and $d = (w',w'')\in D$ be sent to the 
basis elements $y_1$ and $x_1$ in $M$.
Then $w'' - w, w'\in V$.
Now $w' \in Cw'''$ for some $w'''\in V$,
and $W = V\oplus Ku$ where $u = w''-w'''$,
and $D = C \oplus K(u,0)$, giving a splitting of the extension.
\end{proof}

\begin{proof}[Proof of Theorem \ref{t:mainrelationthm}]
Let $U$ be $C^\flat$ or $C^\sharp$. We need to show that any map from a finitely presented, so finite dimensional, 
Kronecker module $M$ to the third term in the exact
sequence
\[
0\to (U, C|_U) \to (V,C) \to (V/U,C/C|_U)\to 0
\]
lifts to a map 
to the middle term. It is enough to let $M$ be indecomposable and  show the pullback sequence 
\[
0\to (U,C|_U) \to 
(W,D) \to 
M
\to 0
\]
is split. By Lemma ~\ref{l:straightforward} we have $(C|_{C^\sharp})^\sharp = C^\sharp$ and $ (C|_{C^\flat})^\sharp = C^\flat$, so if $M$ is preprojective, 0-regular or $\infty$-regular then 
the pullback sequence splits by Lemma~\ref{l:somesplit}. Assume instead that $M$ is preinjective or regular automorphic. There is nothing to prove if there are no non-zero maps $M\rightarrow(V,C)$.
By Lemma~\ref{l:nomaps} this means we can assume that 
$U=C^\flat$ and that $M$ is regular automorphic. Hence $D^\flat = C^\flat$
and $D^\sharp=W$ by Lemma~\ref{l:exactprops}, and thus the pullback sequence 
is the exact sequence of Lemma~\ref{l:flatsharpexact} 
for the relation $(W,D)$.
This splits by \cite[Lemma 4.6]{CB}, since the quotient is finite dimensional.
\end{proof}

\begin{proof}[Proof of Corollary \ref{c:maincor}]
Assume that $(V,C)$ is $\Sigma$-pure-injective as a Kronecker module.
By \cite[Corollary 4.4.13]{P} any pure submodule of it is a direct summand.
In particular, by Theorem \ref{t:mainrelationthm}, this applies to $(C^\flat,C|_{C^\flat})$.
Thus also $(C^\flat,C|_{C^\flat})$ is pure-injective.

Since $(C^\flat,C|_{C^\flat})$ is a pure submodule in $(V,C)$,
it is also pure in $(C^\sharp, C|_{C^\sharp})$,
see for example \cite[Lemma 2.1.12]{P}.
Thus the exact sequence of Lemma~\ref{l:flatsharpexact} splits. By Lemma~\ref{l:splitequiv} we have \[
(C^\flat,C|_{C^\flat})\oplus (C^\sharp/C^\flat,(C|_{C^\sharp})/(C|_{C^\flat})) \cong
(C^\sharp,C|_{C^\sharp})
\]
Since $(C^\sharp,C|_{C^\sharp})$ is a pure submodule of 
the $\Sigma$-pure injective module $(V,C)$, $(C^\sharp,C|_{C^\sharp})$ is $\Sigma$-pure injective, 
hence so is $(C^\sharp/C^\flat,(C|_{C^\sharp})/(C|_{C^\flat}))$.
This means the inclusion of Kronecker modules
\[ 
(C^\sharp/C^\flat,(C|_{C^\sharp})/(C|_{C^\flat}))^{(\mathbb{N})} 
\subseteq (C^\sharp/C^\flat,(C|_{C^\sharp})/(C|_{C^\flat}))^{\mathbb{N}} 
\]
splits. Thus the inclusion $(C^\sharp/C^\flat)^{(\mathbb{N})} \subseteq (C^\sharp/C^\flat)^{\mathbb{N}}$ 
of $K[T,T^{-1}]$-modules splits,
so $C^{\sharp}/C^{\flat}$ is a $\Sigma$-pure-injective $K[T,T^{-1}]$-module.
\end{proof}

\section{String algebras}
\label{s:string}

We recall some notation from \cite{CB}.

\subsection*{Words}
(\cite[\S 1]{CB}) A \textit{letter} is either an arrow $x$ or its
formal inverse $x^{-1}$. Let $I$ be one of the sets $\{0,\dots,n\}$
(for some $n\in\mathbb{N}$), $\mathbb{N}$, $-\mathbb{N}$ or $\mathbb{Z}$.
For $I\neq\{0\}$, an $I$-\textit{word} is a sequence of letters 
\[
C=\begin{cases}
C_{1}\dots C_{n} & (\mbox{if }I=\{0,\dots,n\})\\
C_{1}C_{2}\dots & (\mbox{if }I=\mathbb{N})\\
\dots C_{-1}C_{0} & (\mbox{if }I=-\mathbb{N})\\
\dots C_{-1}C_{0}\mid C_{1}C_{2}\dots & (\mbox{if }I=\mathbb{Z})
\end{cases}
\]
(a bar $\mid$ shows the position of $C_{0}$ and $C_{1}$ when $I=\mathbb{Z}$)
satisfying:
\begin{itemize}
\item[(a)]
if $C_{i}$ and $C_{i+1}$ are consecutive letters, then the tail
of $C_{i}$ is equal to the head of $C_{i}$. 
\item[(b)]
if $C_{i}$ and $C_{i+1}$ are consecutive letters, then $C_{i}^{-1}\neq C_{i+1}$ 
\item[(c)]
no zero relation $x_{1}\dots x_{m}\in\rho$, nor its inverse $x_{m}^{-1}\dots x_{1}^{-1}$
occurs as a sequence of consecutive letters in $C$. 
\end{itemize}
For $I=\{0\}$ there are \textit{trivial words} $1_{v,\epsilon}$ for each
vertex $v$ and each $\epsilon=\pm1$. By a \textit{word} we mean an
$I$-word for some $I$. 

The \textit{inverse} $C^{-1}$
of $C$ is defined by inverting its letters (where $(x^{-1})^{-1}=x$)
and reversing their order. By convention $(1_{v,\epsilon})^{-1}=1_{v,-\epsilon}$,
and the inverse of a $\mathbb{Z}$-word is indexed so that $(\dots C_{0}\mid C_{1}\dots)^{-1}=\dots C_{1}^{-1}\mid C_{0}^{-1}\dots$

If $C$ is a $\mathbb{Z}$-word and $n\in\mathbb{Z}$, the \textit{shift} $C[n]$
is the word $\dots C_{n}\mid C_{n+1}\dots$ We say that a word $C$
is \textit{periodic} if it is a $\mathbb{Z}$-word and $C=C[n]$ for
some $n>0$. The minimal such $n$ is called the \textit{period}.
We extend the shift to $I$-words $C$ with $I\neq\mathbb{Z}$ by
defining $C[n]=C$.

\subsection*{Modules given by words} 
For any $I$-word $C$ and any $i\in I$
there is an associated vertex $v_{i}(C)$, the tail of $C_{i}$ or
the head of $C_{i+1}$, or $v$ for $C=1_{v,\epsilon}$. Given an
$I$-word $C$ let $M(C)$ be the $\Lambda$-module generated by the
elements $b_{i}$ subject to the relations 
\[
e_{v}b_{i}=\begin{cases}
b_{i} & (\mbox{if }v_{C}(i)=v)\\
0 & (\mbox{otherwise})
\end{cases}
\]
for any vertex $v$ in $Q$ and 
\[
xb_{i}=\begin{cases}
b_{i-1} & (\mbox{if }i-1\in I\mbox{ and }C_{i}=x)\\
b_{i+1} & (\mbox{if }i+1\in I\mbox{ and }C_{i+1}=x^{-1})\\
0 & (\mbox{otherwise})
\end{cases}
\]
for any arrow $x$ in $Q$. 
Given a periodic $\mathbb{Z}$-word $C$
of period $p$, and a $K[T,T^{-1}]$-module $V$, there is an automorphism
of the underlying vector space of $M(C)$ given by $b_{i}\mapsto b_{i-p}$. 
Hence $M(C)$ is a $\Lambda$-$K[T,T^{-1}]$-bimodule and we let $M(C,V)=M(C)\otimes_{K[T,T^{-1}]}V$.

By a \emph{string module} we mean a module of the form $M(C)$ where $C$ is not a periodic $\mathbb{Z}$-word. 
By a \emph{band module} we mean a module of the form $M(C,V)$ where $C$ is a periodic $\mathbb{Z}$-word and $V$ is an indecomposable $K[T,T^{-1}]$-module. 

\subsection*{Sign, heads and tails} 
(\cite[\S 2]{CB}) We choose a \textit{sign} $\epsilon=\pm1$
for each letter $l$, such that if distinct letters $l$ and $l'$
have the same head and sign, then $\{l,l'\}=\{x^{-1},y\}$ for some
zero relation $xy\in\rho$.

The head of a finite word or $\mathbb{N}$-word
$C$ is defined to be $v_{0}(C)$, so it is the head of $C_{1}$,
or $v$ for $C=1_{v,\epsilon}$. The sign of a finite word or $\mathbb{N}$-word
$C$ is defined to be that of $C_{1}$, or $\epsilon$ for $C=1_{v,\epsilon}$.

For $v$ a vertex and $\epsilon=\pm1$, we define $\mathcal{W}_{v,\epsilon}$
to be the set of all $I$-words with head $v$, sign $\epsilon$, and where $I\subseteq \mathbb{N}$.

\subsection*{Composing words} 
The composition $CD$ of a word $C$ and
a word $D$ is obtained by concatenating the sequences of letters,
provided that the tail of $C$ is equal to the head of $D$, the words
$C^{-1}$ and $D$ have opposite signs, and the result is a word.

By convention $1_{v,\epsilon}1_{v,\epsilon}=1_{v,\epsilon}$ and the
composition of a $-\mathbb{N}$-word $C$ and an $\mathbb{N}$-word
$D$ is indexed so that $CD=\dots C_{0}\mid D_{1}\dots$ If $C=C_{1}\dots C_{n}$
is a non-trivial finite word and all powers $C^{m}$ are words, we
write $C^{\infty}$ and $^{\infty}C^{\infty}$ for the $\mathbb{N}$-word
and periodic word $C_{1}\dots C_{n}C_{1}\dots C_{n}\dots$ and $\dots C_{n}\mid C_{1}\dots$
If $C$ is an $I$-word and $i\in I$, there are words $C_{>i}=C_{i+1}C_{i+2}\dots$
and $C_{\leq i}=\dots C_{i-1}C_{i}$ with appropriate conventions
if $i$ is maximal or minimal in $I$, such that $C=(C_{\leq i}C_{>i})[i]$.

\subsection*{Relations given by words} 
(\cite[\S 4]{CB}) If $M$ is a $\Lambda$-module
and $x$ is an arrow with head $v$ and tail $u$, then multiplication
by $x$ defines a linear map $e_{u}M\rightarrow e_{v}M$, and hence
a linear relation from $e_{u}M$ to $e_{v}M$.

By composing such relations
and their inverses, any finite word $C$ defines a linear relation
from $e_{u}M$ to $e_{v}M$, where $v$ is the head of $C$ and $u$
is the tail of $C$. We denote this relation also by $C$.

Thus, for
any subspace $U$ of $e_{u}M$, one obtains a subspace $CU$ of $e_{v}M$.
We write $C0$ for the case $U=\{0\}$ and $CM$ for the case $U=e_{u}M$.

\subsection*{Filtrations given by words} 
(\cite[\S 6]{CB}) For $C\in\mathcal{W}_{v,\epsilon}$ and any $\Lambda$-module $M$ define subspaces $C^{-}(M)\subseteq C^{+}(M)\subseteq e_{v}M$
as follows.

Suppose $C$ is finite. Let $C^{+}(M)=Cx^{-1}0$
 if there is an arrow $x$ such that $Cx^{-1}$ is a word, and otherwise
$C^{+}(M)=CM$. Similarly let $C^{-}(M)=CyM$ if there is an arrow $y$
such that $Cy$ is a word, and otherwise $C^{-}(M)=C0$.

If instead $C$ is an $\mathbb{N}$-word let $C^{+}(M)$ be the set of
$m\in M$ such that there is a sequence $m_{n}$ $(n\geq0)$ with $m_{0}=m$ and $m_{n-1}\in Cm_{n}$ for all $n\geq1$, and define
$C^{-}(M)$ to be the set of $m\in M$ such that there is a sequence
$m_{n}$ as above which is eventually zero.

\subsection*{Subgroups of finite definition}
(\cite[\S 1.1.1]{P}) A \textit{pp-definable subgroup} of $M$ is an additive subgroup of $M$ of the form
\[
\begin{array}{c}
\{m\in M\mid A\underline{m}= 0 \textnormal{ for some }  \underline{m}= \left( 
\begin{array}{c} m_0 \\ \vdots \\ m_{c-1}  \end{array} \right)
\in M^{c} \textnormal{ with } m=m_{0} \} \end{array}
\]
where  $r,c\geq 1$ and $A=(a_{ij})$ is a matrix in $\mathbb{M}_{r,c}(\Lambda)$. If $r=c=1$ and $A=a$ this gives
$\{m\in M\mid am= 0 \}$. If $r=1$, $c=2$, and $A=(\begin{array}{cc} -1 & a \end{array})$ this gives
$\{m\in M\mid \exists m'\in M \textnormal { such that } m= am' \}$. If $C$ is a finite word then $CM$ is a pp-definable subgroup of $M$ (see \cite[\S 5.3.2]{Har}, \cite[Example 1.1.2]{P} or \cite[\S 4]{PP2}).

\begin{lem}
\label{lem:CMisppdef}If $M$ is a pure-injective $\Lambda$-module and $C$ is
an $\mathbb{N}$-word then $C^{+}(M)=\bigcap_{n\geq0}C_{\leq n}M$.
\end{lem}

\begin{proof}
Clearly $C^{+}(M)\subseteq\bigcap_{n\geq0}C_{\leq n}M$
so it suffices to pick $m\in M$ such that $m\in C_{\leq n}M$
for all $n\geq0$ and show $m\in C^{+}(M)$. Suppose, for an arbitrary but fixed $i>0$, we can choose
$m_{i-1}\in\bigcap_{n\geq i}(C_{\geq i})_{\leq n}M$. For $n>i$ the
set $\Delta_{n}=C_{i}^{-1}m_{i-1}\cap(C_{>i})_{\leq n}M$
is a non-empty coset of a pp-definable subgroup. We have
$\bigcap_{s\in S}\Delta_{s}=\Delta_{\text{max\,}S}\neq\emptyset$ for any finite subset $S$ of $\{n\in\mathbb{N}\mid n>i\}$,
so as $M$ is algebraically compact there exists $m_{i}\in\bigcap_{n>i}(C_{>i})_{\leq n}M$
such that $m_{i-1}\in C_{i}m_{i}$ (see \cite[\S 4.2.1]{P}). Setting $m_{0}=m$ gives the required sequence $m_{0},m_{1},m_{2},\dots\in M$.
\end{proof}

\subsection*{Refined functors}
(\cite[\S 7]{CB}) If $(B,D)\in\mathcal{W}_{v,1}\times\mathcal{W}_{v,-1}$ and $M$ is a $\Lambda$-module, let $F_{B,D}(M)=F_{B,D}^{+}(M)/F_{B,D}^{-}(M)$ where 
\[
\begin{array}{c}
F_{B,D}^{+}(M)=B^{+}(M)\cap D^{+}(M)\text{ and }\\
F_{B,D}^{-}(M)=(B^{+}(M)\cap D^{-}(M))+(B^{-}(M)\cap D^{+}(M)).
\end{array}
\]
If $(B,D)\in\mathcal{W}_{v,1}\times\mathcal{W}_{v,-1}$
and $C=B^{-1}D$ is a periodic word, say $D=E^{\infty}$ and $B=(E^{-1})^{\infty}$ for some finite word $E$, then $F_{B,D}^{+}(M)=E^{\sharp}$, $F_{B,D}^{+}(M)=E^{\flat}$ and the linear relation $E$ on $e_{v}M$ induces an automorphism of $F_{B,D}(M)$ (see~\S1). Hence $F_{B,D}$ defines a functor from $\Lambda$-modules to $K[T,T^{-1}]$-modules. Otherwise $C$ is a non-periodic word and we consider $F_{B,D}$ as a functor from the category of $\Lambda$-modules to $K$-vector spaces. 

In general there is a natural isomorphism between $F_{B,D}$ and the functor $G_{B,D}$ defined by $G_{B,D}(M)=G_{B,D}^{+}(M)/G_{B,D}^{-}(M)$ for any $\Lambda$-module $M$ where $G_{B,D}^{\pm}(M)=B^{+}(M)+D^{\pm}(M)\cap B^{-}(M)$.

\begin{cor}
\label{cor:pressurj}
Let $\theta:N\rightarrow M$ be a homomorphism of $\Lambda$-modules where is pure-injective $\Lambda$.
If $F_{B,D}(\theta)$ is surjective for all $(B,D)\in\bigcup_{v}\mathcal{W}_{v,1}\times\mathcal{W}_{v,-1}$
then $\theta$ is surjective. 
\end{cor}
\begin{proof}
For the contrapositive we suppose $\text{im}(\theta)\neq M$, and
so we can choose a vertex $v$ and some element $m\in e_{v}M\setminus e_{v}\text{im}(\theta)$.
The set $S=e_{v}\text{im}(\theta)+m$ contains $m$ but
not $0$, so by combining lemma \ref{lem:CMisppdef} (ii) and \cite[Lemma 10.3]{CB},
there is a word $B\in\mathcal{W}_{v,\epsilon}$ such that $S$ meets
$B^{+}(M)$ but not $B^{-}(M)$. Following the proof of \cite[Lemma 10.5]{CB}
we have that $S$ meets $G_{B,D}^{+}(M)$ but
not $G_{B,D}^{-}(M)$ for some $(B,D)\in\mathcal{W}_{v,1}\times\mathcal{W}_{v,-1}$.
Following the second half of the proof of \cite[Lemma 10.6]{CB},
this shows $G_{B,D}(\theta)$ is not surjective. 
\end{proof}

\begin{proof}[Proof of Theorem \ref{t:mainstringtheorem}]
We show that every $\Sigma$-pure-injective $\Lambda$-module $M$ is a direct sum of string modules $M(C)$
and band modules $M(C,V)$ with $V$ $\Sigma$-pure-injective.

If $(B,D)\in\mathcal{W}_{v,1}\times\mathcal{W}_{v,-1}$
and $C=B^{-1}D$ is periodic, say $D=E^{\infty}$ and $B=(E^{-1})^{\infty}$ for some finite word $E$, then $(e_{v}M,E)$ is split by Corollary \ref{c:maincor}.
Following the proof of \cite[Theorem 9.2]{CB}, this means there is a homomorphism $\theta:N\rightarrow M$ where $N$ is a direct sum of string and band modules, and $F_{B,D}(\theta)$ is an isomorphism for all pairs of words $(B,D)\in\mathcal{W}_{v,1}\times\mathcal{W}_{v,-1}$
such that $C=B^{-1}D$ is a word. By \cite[Lemma 9.4]{CB} this means $\theta$ is injective, and $\theta$ is surjective by Corollary \ref{cor:pressurj}. 
\end{proof}

Note that any $\Sigma$-pure-injective module is a direct sum of indecomposables, 
but conversely not every direct sum of
indecomposable $\Sigma$-pure-injective modules is $\Sigma$-pure-injective,
see for example \cite[Example 4.4.18]{P}.

Ringel has shown that $M(C)$ is $\Sigma$-pure-injective provided $C$ is a so-called
\emph{contracting} word \cite[\S 5]{R3}. A more general result is due to Harland \cite{Har}.

\subsection*{Harland's criterion}
For each vertex $v$ and each $\epsilon \in \{ \pm 1 \} $ there is a 
total ordering $<$ on $\mathcal{W}_{v,\epsilon}$ given by $C<C'$ if
\begin{itemize}
\item[(a)] $C=ByD$ and $C'=Bx^{-1}D'$ for arrows $x$ and $y$ and words $B$, $D$, and $D'$ (with $B$ finite),
\item[(b)] $C'$ is finite and $C=C'yD$ for an arrow $y$ and a word $D$, or
\item[(c)] $C$ is finite and $C' = Cx^{-1}D$ for an arrow $x$ and a word $D$.
\end{itemize}

For any $I$-word $C$ and any $i\in I$ the words $C_{>i}$ and $(C_{\leq i})^{-1}$ have 
the same head but opposite signs. Let $C(i,\pm 1)$ be the one with sign $\pm 1$. 
The following result is \cite[Proposition 14 and Theorem 42]{Har}. 
(Note that Harland uses the opposite ordering on $\mathcal{W}_{v,\epsilon}$ 
so has the ascending chain condition.)

\begin{pro}
Let $\Lambda$ be  finite dimensional and $C$ be an $I$-word. 
Then $M(C)$ is $\Sigma$-pure-injective
if and only if for each vertex $v$ and each $\epsilon \in \{ \pm 1 \}$ 
every descending chain in $\{ C(i, \epsilon) : i\in I, v_{i}(C)=v \}$ stabilizes.
\end{pro}

On page 243 of \cite[\S 6.9]{Har} there is an example of an aperiodic word $C$ where $M(C)$ is pure-injective.

\medskip
\textit{Acknowledgement.}
The first listed author is grateful to Rosanna Laking for many helpful conversations 
about string algebras and the model theory of modules.

\end{document}